\documentclass[a4paper,11pt]{amsart}

\usepackage[pctex32]{graphics}
\usepackage[latin1]{inputenc}
\usepackage{epsfig}
\usepackage[english]{babel}
\usepackage{graphicx,colortbl}
\usepackage{graphics}
\usepackage[dvips]{geometry}
\usepackage{psfrag}
\usepackage{graphicx}
\usepackage[active]{srcltx}
\usepackage{amscd}
\usepackage{amsmath,amstext, amsthm}
\usepackage{amsfonts}
\usepackage{amssymb}
\usepackage{epsfig} 

\usepackage{graphicx,color}
\usepackage[colorlinks,linkcolor=green]{hyperref}

\newtheorem{theorem}{Theorem}[section]

\theoremstyle{definition}

\numberwithin{equation}{section}
\newcommand{\be}{\begin{equation}}
\newcommand{\ee}{\end{equation}}

\newcommand{\vv}{{\bf V}}

\newcommand{\ra}{\rangle}

\begin{document}

\title[Linearizability and  critical period bifurcations of  Riccati   system]
{Linearizability and  critical period bifurcations of
\\
 a generalized Riccati system}
\author[V.G. ROMANOVSKI,  W. FERNANDES,   Y. TANG, Y. TIAN]
{Valery  G. Romanovski$^{1,2,3,4}$, Wilker Fernandes$^{5}$,  Yilei Tang$^{6,4}$ and Yun Tian$^{1}$ }

\address{$^1$
Department of Mathematics, Shanghai Normal University, Shanghai, 200234, P.R. China}
\email{Valery.Romanovski@uni-mb.si (V.G. Romanovski), ytian22@shnu.edu.cn (Y. Tian)}

\address{$^2$ Faculty of Electrical Engineering and Computer Science, University of Maribor, Smetanova 17,  Maribor, SI-2000 Maribor, Slovenia}

\address{$^3$ Faculty of Natural Science and Mathematics, University of Maribor, Koro\v ska c.160, Maribor, SI-2000 Maribor, Slovenia}

\address{$^4$ Center for Applied Mathematics and Theoretical Physics, University of Maribor, Krekova 2, Maribor,  SI-2000 Maribor, Slovenia}

\address{$^5$ Instituto de Ci\^encias Matem\'aticas e de Computa\c{c}\~ao - USP, Avenida Trabalhador S\~ao-carlense, 400, 13566-590, S\~ao Carlos, Brazil}
\email{wilker.thiago@usp.br (W. Fernandes)}

\address{$^6$ School of Mathematical Science, Shanghai Jiao Tong University, Dongchuan Road 800, Shanghai, 200240, P.R. China}
\email{Corresponding author. mathtyl@sjtu.edu.cn (Y. Tang)}


\date{}

\begin{abstract}
In this paper we  investigate the isochronicity and linearizability problem for a cubic polynomial differential system
which can be considered as a generalization of the Riccati system.
Conditions for  isochronicity and linearizability are found.
 The global structure  of systems of the family with  an isochronous center is determined.
 Furthermore, we find the order of weak center and study the problem of local bifurcation of critical periods in a neighborhood of the center.
  \end{abstract}


\maketitle

\section{Introduction}

A classical problem
in the qualitative theory of ordinary differential equations is to characterize the existence  of centers and  isochronous centers.
A singular point of a planar autonomous differential system is called a {\it center} if all solutions sufficiently closed to it are periodic,
 that is, all trajectories in a small neighborhood of the singularity are ovals.
If all  periodic solutions inside the period annulus of the center have the same period it is said that the center is {\it isochronous}.

Poincar\'e and Lyapunov have shown that the existence of an isochronous center at the origin of a system of the form
\begin{equation} \label{sys introd 1}
\dot{x} = -y + P(x,y),
\qquad
\dot{y} = x + Q(x,y),
\end{equation}
where $P(x,y)$ and $ Q(x,y)$ are real polynomials without constant and linear terms, is  equivalent to the linearizability of the system.
This equivalence has made the studies of the isochronicity problem simpler, since the linearizability problem can be extended to the complex field, where the computational methods are more efficient.

The investigation on isochronicity of oscillations  started in the 17th century, when Huygens studied the cycloidal pendulum \cite{Huy}.
However, only in the second half of the last century the isochronicity problem began to be intensively studied.
In 1964  Loud \cite{Loud} found the necessary and sufficient conditions for isochronicity  of system \eqref{sys introd 1}
with $P$ and $Q$ being quadratic homegeneous polynomials.
Later on, the isochronicity problem was solved for system \eqref{sys introd 1} when $P$ and $Q$
are homogeneous polynomials of degree three \cite{P} (see also \cite{Li Jibin}) and degree five \cite{R-C-H}.
However in the case of the linear center perturbed by homogeneous polynomials of degree four
the problem is still unsolved, although some partial results were obtained \cite{C-G-G-2,GKLR}. The reason is that linearizability quantities (which are polynomials
in the parameters of system \eqref{sys introd 1} defined at the beginning
of Section 2)
have more complicate expressions in the case of homogeneous  perturbations  of degree four, than in the case of  homogeneous  perturbations  of degree five.
 There are also many works devoted to the investigation of particular families
of some other   polynomial systems, see e.g. \cite{AmLS1982, C-G-G-1, CHRZ, C-G-M-M, M-R-T, R-S} 
and references therein.
Many works also deal with investigation of isochronicity of Hamiltonian systems, see e.g. \cite{ChrisD1997, Cim-G-M2000, Ha1983, JarV2002, LliR2015}
and references given there.

The problem of critical period   bifurcations is tightly related to the isochronicity  problem.
In a neighborhood of a center the so-called {\it period function} $T(r)$  gives the least period of
the periodic solution passing through the point with coordinates $(x, y) =(r, 0)$ inside the period annulus of the center.
For a center that is not isochronous any value $r>0$ for which $T'(r) =0$ is called a critical period.
The problem of critical period   bifurcations  is aimed on estimating of the number of critical periods that can arise near the center under small perturbations.
In 1989,  Chicone and Jacobs \cite{CJ1989} introduced for the first time the theory of local bifurcations of critical periods
and solved the problem  for the quadratic system.
Local bifurcations of critical periods have been investigated
for cubic  systems with homogeneous nonlinearities \cite{RT1993}, the reduced
 Kukles system \cite{RT1997}, 
the Kolmogorov system \cite{CHRZ2014}, the $\mathrm{Z}_2$-equivariant   systems \cite{CHR2014}
and some other families (see e.g. \cite{CGS2008, GZ2008,YuH2009} and references therein).
In \cite{FerLRS2015} a general approach to studying bifurcations of critical periods based on a complexification of the system
 was described, and some upper bounds on the number of critical periods of several cubic systems were obtained.

In this paper we are interested in
the family of
Riccati systems.
The classic {\it Riccati system} is written
in the form
\begin{equation} \label{sys-Ric}
\dot{x} =  1, ~~~~
\dot{y} = g_2(x)y^2 + g_1(x)y + g_0(x),
\end{equation}
where  each $g_j(x)$ is a $\mathcal{C}^1$ function  with respect to   $x$
and $g_2(x)g_0(x)\not\equiv0$.
System \eqref{sys-Ric} becomes a special case of Berouilli system if $g_0(x)\equiv0$, and it obviously is a linear differential system if $g_2(x)\equiv0$.

The Riccati  equation  has been   invstigated by many authors, see for example \cite{Lli2014, Lli2015} and references therein.
They are important since they can be used to solve second-order ordinary differential equations and can be applied
in studying the third-order Schwarzian differential \cite{Neh1975}.
It also has many applications in both physics and mathematics.
For instance,   renormalization group equations for running coupling
constants in quantum field theories \cite{Buc1992}, nonlinear physics \cite{Mat1991},   Newton's laws of motion \cite{Now2002},
thermodynamics \cite{Ros2002} and  variational calculus \cite{Zel1998}.

Recently Llibre and Valls \cite{Lli2014, Lli2015} investigated  the planar differential system
\begin{equation*} 
\dot{x} =  f(y),~~~~
\dot{y} = g_2(x)y^2 + g_1(x)y + g_0(x),
\end{equation*}
which is called the {\it  generalized Riccati system}, since it becomes the classic   Riccati system when $f(y)\equiv 1 $.
In this paper we study a subfamily of the generalized Riccati system,   cubic  systems of the form
\begin{equation} \label{sys-Ric3}
\begin{aligned}
\dot{x}& =     -y + a_{02} y^2  +  a_{03} y^3,
\\
\dot{y} &=(b_{02}+ b_{12} x) y^2 + (b_{11} x + b_{21} x^2 ) y + (x + b_{20} x^2 +   b_{30} x^3)
    \\
    &= x +  b_{20} x^2 + b_{11} x y + b_{02} y^2  +   b_{30} x^3 + b_{21} x^2 y +
    b_{12} x y^2 ,
  \end{aligned}
\end{equation}
where $x, y $  are unknown real functions
and   $a_{ij}, b_{ij}$ are real parameters.
Note that system \eqref{sys-Ric3} is the so-called reduced Kukles system when $a_{02} = a_{03}=0$. 

The aims of our study are to obtain conditions on parameters   $a_{ij}$ and $b_{ij}$ for the linearizability  of system \eqref{sys-Ric3},
to study the global structures of trajectories when the system has an isochronous center, and to investigate the local bifurcations of critical periods
at the origin.
In Section \ref{Sec:Results} we present our main result on linearizability, Theorem \ref{Theorem 1}, which gives conditions for the linearizability of system \eqref{sys-Ric3}. We also describe an approach for deriving such conditions
which is based on making use of  modular computations which are performed in the
systems of computer algebra {\sc Singular} \cite{sing} and {\sc Mathematica} \cite{WM}.
The approach can be applied to investigate many problems involving solving
systems of algebraic polynomials.
In Section \ref{Sec:global} we study the global dynamics of system \eqref{sys-Ric3}  when the origin is an isochronous center.
The last section  is devoted to the investigation of local bifurcations  of critical periods in a neighborhood of the center.


\bigskip

\section{Linearizability  of system \eqref{sys-Ric3}} \label{Sec:Results}


We first briefly remind an approach for  studying the isochronicity and linearizability problems  for  polynomial differential systems of the form
\begin{equation}\label{System general real}
\dot{x} = -y + \sum_{p + q = 2}^n a_{p,q} x^{p} y^{q}, 
\hspace{0.3cm}
\dot{y} = x + \sum_{p + q = 2}^n b_{p,q} x^{p} y^{q}, 
\end{equation}
where $x,y $ and $a_{p,q}, b_{p,q}$ are in   $\mathbb{R}$. 

System \eqref{System general real} is  \textit{linearizable}
if there is an analytic change of coordinates
\begin{equation}\label{lin equa}
x_1 = x +\sum_{m+n \geq 2} c_{m,n} x^m y^n, \hspace{0.5cm} y_1 = y + \sum_{m+n \geq 2} d_{m,n} x^m y^n,
\end{equation}
which reduces \eqref{System general real} to the canonical linear system
$\dot{x}_1= -y_1$, $\dot{y}_1= x_1$.

Obstacles for existence of a transformation \eqref{lin equa} are some polynomials in  parameters  of system \eqref{System general real}
called the \textit{linearizability quantities} and denoted by $i_k, j_k$ ($k=1,2, ...$).

%
%

Differentiating with respect to $t$ both sides of each equation of \eqref{lin equa} we obtain
\begin{equation}\label{diff lin}
\begin{aligned}
\dot{x}_1 = & \dot{x} + \left( \sum_{m+n \geq 2} m c_{m,n} x^{m-1} y^n \right) \dot{x} + \left( \sum_{m+n \geq 2} n c_{m,n} x^m y^{n-1} \right) \dot{y},
\\
\dot{y}_1 = & \dot{y} + \left( \sum_{m+n \geq 2} m d_{m,n} x^{m-1} y^n \right) \dot{x} + \left( \sum_{m+n \geq 2} n d_{m,n} x^m y^{n-1} \right) \dot{y}.
\end{aligned}
\end{equation}
Substituting in the above equations the expressions from  \eqref{lin equa} and \eqref{System general real},
one computes the linearizability quantities $i_k, j_k$ step-by-step (see e.g.  \cite{FernR2017} for more details).


From \eqref{diff lin} it is easy to see that the linearizability quantities
$i_k, j_k$ are polynomials in   parameters $a_{p,q}, b_{p,q}$ of system \eqref{System general real}. We denote by $(a,b)$ the $s$-tuple ($s$ is the number of  parameters $a_{p,q}, b_{p,q}$ in system \eqref{System general real}) of parameters of \eqref{System general real}, so $(a,b)=(a_{2,0},a_{1,1}, \dots, b_{0,n})$,
and by $\mathbb{R}[a,b]$ and $\mathbb{C}[a,b]$ the rings
of polynomials in  $a_{p,q}, b_{p,q}$ with real and complex coefficients, respectively.

Thus, the simultaneous vanishing of all linearizability quantities $i_k, j_k$ provides conditions which characterize when a system of the form \eqref{System general real} is linearizable.
The ideal defined by the linearizability quantities, $\mathcal{L} = \langle i_1, j_1, i_2, j_2, ...  \rangle \subset \mathbb{R}[a,b]$, is called the   \textit{linearizability ideal}
and its affine variety, $V_{\mathcal{L}} = {\bf V} (\mathcal{L})$  is called the \textit{linearizability variety}.


In order to  find a linearizing change of coordinates explicitly  one can look for Darboux linearization.
To construct a Darboux linearization for system \eqref{System general real}
it is convenient to complexify the system using the substitution
\begin{equation} \label{zw}
z=x+iy, \qquad w=x-iy.
\end{equation}
Then, after a time rescaling by $i$ we obtain from \eqref{System general real} a system of the form
\begin{equation} \label{System general complex-1}
\dot{z}= z + X(z,w) , \qquad \dot{w}= - w - Y(z,w).
\end{equation}
System \eqref{System general real}
is linearizable if and only if system  \eqref{System general complex-1}
is linearizable.

A \textit{Darboux factor} of system \eqref{System general complex-1} is a polynomial $f(z,w)$ satisfying
\begin{equation*}
\dfrac{\partial f}{\partial z} \dot{z} + \dfrac{\partial f}{\partial w} \dot{w} = K f,
\end{equation*}
where polynomial $K(z,w)$ is  called the \textit{cofactor of $f$}.
A \textit{Darboux linearization} of system \eqref{System general complex-1} is an analytic change of coordinates  $z_1 =  Z_1 (z,w)$,
$w_1 = W_1 (z,w)$, such that
\begin{equation*}
\begin{aligned}
Z_1(z,w) = \prod_{j=0}^{m} f_j^{\alpha_j}(z,w) = z + \tilde{Z}_1 (z,w),
\\
W_1(z,w) = \prod_{j=0}^{n} g_j^{\beta_j}(z,w) = w + \tilde{W}_1 (z,w),
\end{aligned}
\end{equation*}
which linearizes \eqref{System general complex-1},
where $f_j, g_j \in \mathbb{C}[z,w]$, $\alpha_j, \beta_j \in
\mathbb{C}$, and $\tilde{Z}_1$ and $\tilde{W}_1$ have neither constant terms nor
linear terms.


It is easy to see that system \eqref{System general complex-1} is Darboux linearizable if there exist $s+1 \geq 1$ Darboux factors $f_0,...,f_s$ with corresponding cofactors $K_0,...,K_s$, and $t+1 \geq 1$ Darboux factors $g_0,...,g_t$ with corresponding cofactors $L_0,...,L_t$  with the following properties:
\begin{enumerate}
	\item[(i)] $f_0(z,w) = z +  \cdot \cdot \cdot \mbox{ but }  f_j(0,0) = 1 \mbox{ for } j \geq 1$;
	\item[(ii)] $g_0(z,w) = w +  \cdot \cdot \cdot \mbox{ but }   g_j(0,0) = 1 \mbox{ for }  j \geq 1$; and
	\item[(iii)] there are $s+t$ constants $\alpha_1,...,\alpha_s, \beta_1,...,\beta_t \in \mathbb{C}$ such that
\begin{equation}  \label{cond 1}
K_0 + \alpha_1 K_1 + \cdot \cdot \cdot + \alpha_s K_s = 1
\hspace{0.3cm} \mbox{  and  } \hspace{0.3cm}
L_0 + \beta_1 L_1 + \cdot \cdot \cdot + \beta_t L_t = -1.
\end{equation}
\end{enumerate}

The Darboux linearization is then given by the transformations
\begin{equation*}
z_1 = H_1(z,w) = f_0 f_1^{\alpha_1} \cdot \cdot \cdot f_s^{\alpha_s},
\qquad
y_1 = H_2(z,w) = g_0 g_1^{\beta_1} \cdot \cdot \cdot g_t^{\beta_t} .
\end{equation*}
The readers can consult \cite{CR2001, M-R-T, R-S} for more details.

Before passing to the results of our paper we remind some
fact about solutions of systems of nonlinear polynomial
equations which we will need for our study.

Denote by $ k[x_1,\dots,x_n] $
the ring of polynomials  with coefficients in a field $k$
and
consider a system of polynomials of   $ k[x_1,\dots,x_n] $:
\begin{eqnarray}\label{s1}
f_1(x_1,\dots,x_n)&=&0,\nonumber\\
\qquad\quad\ \vdots&&\\
f_m(x_1,\dots,x_n)&=&0.\nonumber
\end{eqnarray}

We recall that the ideal $I $ in  $ k[x_1,\dots,x_n] $
 generated by polynomials $f_1,\dots, f_m $, denoted by
$I= \langle f_1, \dots, f_m \ra $, is the set of all polynomials of  $ k[x_1,\dots,x_n] $
expressed in the form $f_1 h_1+f_2 h_2+\dots+ f_m h_m$, where $h_1, h_2, \dots, h_m$
are polynomials of  $ k[x_1,\dots,x_n] $.    The variety
of the ideal  $I=\langle f_1,\dots, f_m \rangle \subset k[x_1,\dots,x_n] $ in $k^n$,
 denoted by $\vv(I)$,   is the zero set of all  polynomials of $I$,
$$
\vv(I)=\left\{ A=(a_1,\dots,a_n) \in k^n  | f(A)=0 \quad {\rm for \ all} \ f\in I
\right\}.
$$
The situation when the variety of a polynomial ideal consists of a
finite number of points
 arises very rarely. In a generic  case, the variety consists of infinitely
many points, so generally speaking,``to solve" system  (\ref{s1}) means to find a  decomposition
of the variety of the ideal into  irreducible components.
More precisely, an affine variety $V \subset k^n $
is \emph{irreducible} if, whenever $V = V_1 \cup V_2$ for affine
varieties $V_1$ and $V_2$, then either $V_1 = V$ or $V_2 = V$.
Let $I$ be an ideal and $V=\vv(I)$ its variety. Then   $V $
 can be represented as a union of irreducible components,
$
V = V_1 \cup \dots \cup
V_m,
$
where each  $V_i$ is irreducible.
The  radical of $I$ denoted by $\sqrt{I}$ is
the set of all polynomials $f$  of  $k[x_1,\dots,x_n] $
such that for some non-negative integer  $p$ $f^p$ is in $I$.
Clearly,  $I$ and $\sqrt{I}$ have the same
varieties.
It is known that $\sqrt{I}$ can be expressed as  an intersection of prime ideals,
$
\sqrt{I} = \cap_{j = 1}^s Q_j.
$
Prime ideals
$Q_i$ are   called \emph{the minimal associate primes of $I$}.
Let
$V_i$ ($i=1,\dots, s$) be the variety of $Q_i$.
Since the variety of an  intersection of some ideals
is equal to the union of the varieties of the ideals,
we have that $\vv(I)=\vv(\sqrt{I})=\cap_{j=1}^{s} V_j$.
For example,  if  $I=\langle x^2y^3, xz^5\ra$, then
$\sqrt{I}=\langle  x y, x z \ra
= \langle x \ra \cap \langle y, z\ra $, that is,
the variety of $I$ is the union of two irreducible
components:  the plane $x=0$ and the line $ y=z=0$.
In the computer algebra system {\sc Singular} \cite{sing} one can compute
the minimal associate primes of a given polynomial
ideal and, thus, the irreducible decomposition of
its variety  using the routine \texttt{minAssGTZ}.

Proceeding now to the results of our paper we first state
the following theorem on
 the linearizability  of system \eqref{sys-Ric3}.

\begin{theorem} \label{Theorem 1}
System \eqref{sys-Ric3}  is linearizable at the origin if one of the following conditions holds:
\begin{enumerate}
	\item[(1)] $b_{12} = a_{02} =  b_{30} =  b_{21} =  a_{03} =  b_{02} + b_{20} = b_{11}^2 + 4 b_{20}^2 = 0$,
	\item[(2)] $b_{12} = a_{02} =  b_{20} =  b_{02} =  b_{21} =  a_{03} =  9 b_{30} - b_{11}^2 =0 $,
 	\item[(3)] $b_{12} =  a_{02} =  b_{11} =  b_{20} =  b_{30} =  b_{21} =  9 a_{03} + 4 b_{02}^2 = 0$,
 	\item[(4)] $ b_{12} =  b_{30} =  b_{21} =  a_{03} =   2 b_{02} + 5 b_{20} = 10 a_{02} - 3 b_{11} = 4 b_{11}^2 + 25 b_{20}^2 = 0$.
\end{enumerate}
\end{theorem}

\begin{proof}
Using the computer algebra system {\sc Mathematica} and the standard procedure mentioned above for system \eqref{sys-Ric3}
we have computed the first eight pairs of the linearizability quantities $i_1,j_1,..., i_8,j_8$.
Their expressions are very large, so we only present the first two pairs in the Appendix.
The reader can easily compute the other quantities using any available computer algebra system
\footnote{
One can download linearizability quantities $i_1$, $j_1$, $\ldots$,
$i_8$, $j_8$ and the {\sc Singular} code to perform the decomposition
of the variety from
\href{http://teacher.shnu.edu.cn/_upload/article/files/79/14/f36e87e342b8b0d6977e6debdeb3/3b818cf4-a6f7-4f07-8669-f4b78e48f733.txt}
{http://teacher.shnu.edu.cn/\textunderscore upload/article/files/79/14/ f36e87e342b8b0d6977e6debdeb3/3b818cf4-a6f7-4f07-8669-f4b78e48f733.txt}.}.

The next computational step is to compute the irreducible decomposition of the variety ${\bf V} (\mathcal{L}_8) = {\bf V} (\langle i_1, j_1,...,i_8, j_8 \rangle )$.

Performing the computations by the routine \texttt{minAssGTZ} \cite{D-L-P-S} of {\sc Singular} \cite{sing} over the field of characteristic 32452843 we obtain that  ${\bf V} (\mathcal{L}_9)$ is equal to the union of the varieties of four ideals.
After lifting these four ideals to the ring of polynomials with rational coefficients using the rational reconstruction algorithm of  \cite{WGD}
we obtain the ideals
\begin{equation*}
\begin{aligned}
J_1 =& \langle b_{12}, a_{02},  b_{30},  b_{21}, a_{03},  b_{02} + b_{20}, b_{11}^2 + 4 b_{20}^2 \rangle,
\\
J_2 =& \langle b_{12}, a_{02}, b_{20},  b_{02},  b_{21},  a_{03},  9 b_{30} - b_{11}^2 \rangle,
\\
J_3 =& \langle b_{12},  a_{02},  b_{11},  b_{20},  b_{30},  b_{21},  9 a_{03} + 4 b_{02}^2 \rangle,
\\
J_4 =& \langle  b_{12},  b_{30},  b_{21},  a_{03},   2 b_{02} + 5 b_{20}, 10 a_{02} - 3 b_{11}, 4 b_{11}^2 + 25 b_{20}^2 \rangle.
\end{aligned}
\end{equation*}
The varieties of $J_1$, $J_2$, $J_3$ and $J_4$  provide conditions $(1)$, $(2)$, $(3)$ and $(4)$ of the theorem, respectively.

To check the correctness of the obtained
conditions we use the procedure described  in \cite{R-P}.
First, we computed the ideal
$ J=  J_1 \cap  J_2 \cap  J_3 \cap J_4 $, which defines the union of all four sets given in the statement of the theorem.
Then we check that  ${\bf V}(J) = {\bf V}(\mathcal{L}_9)$.
 According to the Radical Membership Test, to verify the inclusion  ${\bf V}(J)\supset {\bf V}(\mathcal{L}_9)$
 it is sufficient to check that the Groebner bases of all ideals  $\langle J, 1-w i_k \rangle$,    $\langle J, 1-w j_k \rangle$
 (where $k=1,\dots, 9$ and $w$ is a new variable) computed over $\mathbb{Q}$ are $\{1\}$.
  The computations show that this is the case.
 To check the
 opposite inclusion, ${\bf V}(J)\subset {\bf V}(\mathcal{L}_9)$, it is sufficient to check that  Groebner bases of
  the ideals $ \langle \mathcal{L}_9, 1-w f_i \rangle$ (where the polynomials  $f_i$'s are the polynomials of a basis of $J$)  computed over $\mathbb{Q}$
are equal to  $\{ 1 \}$.
Unfortunately,  we were not able
to perform these computations
over  $\mathbb{Q}$
however we have checked that all the bases are  $\{ 1 \}$ over few fields of finite characteristic.
It yields that the list of conditions in Theorem \ref{Theorem 1} is
the complete list  of linearizability  conditions for system \eqref{sys-Ric3}  with high probability \cite{EA}.

We now prove that under each of conditions
$(1)$--$(4)$ of the theorem the system is linearizable.

$Condition \hspace{0.1cm} (1).$
In this case $b_{11} = \pm 2 b_{20} i$. We consider only the case  $b_{11} =   2 b_{20} i$,
since when  $b_{11} = -2 b_{20} i$  the proof is analogous. After the change of variables \eqref{zw} system \eqref{sys-Ric3} becomes
\begin{equation} \label{sys1-1}
\begin{aligned}
\dot{z}= &  z+  b_{20} z^2,
\\
\dot{w}= &  - w-b_{20} z^2,
\end{aligned}
\end{equation}
which is a quadratic system. By Theorem 3.1 of \cite{CR2001}
 and   Theorem 4.5.1 of    \cite{R-S}
 system \eqref{sys1-1} is Darboux  linearizable and, therefore, system \eqref{sys-Ric3} is linearizable if condition $(1)$ holds.
%

\medskip

$Condition \hspace{0.1cm} (2).$ After substitution \eqref{zw}  system \eqref{sys-Ric3} becomes
 \begin{equation} \label{sys2-1}
\begin{aligned}
\dot{z}= &    z+ \frac{1}{72}( - 18 i b_{11} z^2 +18 i b_{11} w^2 + b_{11}^2 z^3 + 3 b_{11}^2 z^2 w + 3 b_{11}^2 z w^2 + b_{11}^2 w^3),
\\
\dot{w}= &   -w+  \frac{1}{72}(18 i b_{11} z^2-18 i b_{11} w^2 - b_{11}^2 z^3   - 3 b_{11}^2 z^2 w - 3 b_{11}^2 z w^2 - b_{11}^2 w^3).
\end{aligned}
\end{equation}
It has the  Darboux factors
\begin{equation*}
\begin{aligned}
 l_1 =&  z +  \frac{i b_{11}}{12} z^2  + \frac{i b_{11}}{6} z w  + \frac{i b_{11}}{12} w^2,
 \\
 l_2 =&   w -\frac{i b_{11}}{12} z^2- \frac{i b_{11}}{6} z w - \frac{i b_{11}}{12} w^2,
 \\
l_3 =&    1 - \frac{i b_{11}}{6} z+ \frac{b_{11}^2 }{36}z^2 + \frac{i b_{11}}{6} w +\frac{b_{11}^2}{18} z w +\frac{b_{11}^2}{36} w^2,
\\
l_4 =&    1 - \frac{i b_{11}}{3} z +\frac{b_{11}^2}{36} z^2 +\frac{i b_{11}}{3} w +\frac{b_{11}^2}{18} z w +\frac{b_{11}^2}{36} w^2
 \end{aligned}
\end{equation*}
with the respective cofactors
 \begin{equation*}
 \begin{aligned}
 &k_1 =   1 - \frac{i b_{11}}{6} z - \frac{i b_{11}}{6} w,
 ~&k_2 = -1 -\frac{i b_{11}}{6} z -\frac{i b_{11}}{6} w,
  \\
 &k_3 =  -\frac{i b_{11}}{6} z -  \frac{i b_{11}}{6} w,
~&k_4 = -\frac{i b_{11}}{3} z -\frac{ i b_{11}}{3} w.
\end{aligned}
\end{equation*}
It is easy to verify that  \eqref{cond 1} is satisfied
 with  $\alpha_1 =  1, ~\alpha_2 =  -1, ~ \beta_1 = 1$ and $ \beta_2 =  -1$.
 Hence the Darboux linearization for system \eqref{sys2-1}
is given by the analytic change of coordinates
 \begin{equation*}
 z_1 =  l_1 l_3^{\alpha_1} l_4^{\alpha_2},
\hspace{0.3cm}
 w_1 =   l_2 l_3^{\beta_1} l_4^{\beta_2}.
 \end{equation*}
Thus, system \eqref{sys2-1} is linearizable and therefore the corresponding system \eqref{sys-Ric3}  is linearizable as well.

\medskip

$Condition\hspace{0.1cm} (3).$
In this case after substitution \eqref{zw} the corresponding system \eqref{sys-Ric3} is changed to
\begin{equation} \label{sys3-1}
\begin{aligned}
\dot{z}= & z+ \frac{1}{36}(- 9 b_{02} z^2+ 18 b_{02} z w -9 b_{02} w^2 - 2 b_{02}^2 z^3 + 6 b_{02}^2 z^2 w - 6 b_{02}^2 z w^2 + 2 b_{02}^2 w^3),
\\
\dot{w}= & -w+ \frac{1}{36}(9 b_{02} z^2- 18 b_{02} z w+9 b_{02} w^2 - 2 b_{02}^2 z^3 + 6 b_{02}^2 z^2 w - 6 b_{02}^2 z w^2 + 2 b_{02}^2 w^3).
\end{aligned}
\end{equation}
System \eqref{sys3-1} has the Darboux factors
\begin{equation*}
\begin{aligned}
 l_1 =&  z -  \frac{b_{02}}{12} z^2  +  \frac{b_{02}}{6} z w - \frac{b_{02}}{12} w^2, \hspace{0.3cm}
 \\
 l_2 =&   w - \frac{b_{02}}{12} z^2  +  \frac{b_{02} }{6}z w - \frac{b_{02}}{12} w^2,
 \\
 l_3 =&   1 -\frac{2 b_{02} }{3}z +  \frac{2 b_{02}^2}{9} z^2 -\frac{2 b_{02}}{3} w  -
   \frac{4 b_{02}^2}{9} z w +\frac{2 b_{02}^2}{9} w^2,
 \\
l_4 =& 1 - \frac{b_{02}}{3} z + \frac{b_{02}^2}{18} z^2 - \frac{b_{02}}{3} w - \frac{b_{02}^2}{9} z w +
 \frac{ b_{02}^2}{18} w^2,
 \end{aligned}
\end{equation*}
 which allow to construct  the Darboux linearization
 \begin{equation*}
 z_1 =  l_1 l_3^{\alpha_1} l_4^{\alpha_2},
\hspace{0.3cm}
 w_1 =   l_2 l_3^{\beta_1} l_4^{\beta_2},
 \end{equation*}
 where $\alpha_1 = 1,  ~\alpha_2 =  -3, ~ \beta_1 = 1$ and $ \beta_2 =  -3$.


\medskip

$Condition \hspace{0.1cm} (4).$
For this condition  it is easy to see that  $ b_{11} =\pm 5 b_{20}i/2 $.
We consider only the case  $ b_{11} =  5 b_{20}i/2$,
since when  $ b_{11} =-5 b_{20}i/2$  the proof is analogous.
After transformation \eqref{zw} system \eqref{sys-Ric3} becomes
\begin{equation} \label{sys4-1}
\begin{aligned}
\dot{z}= &  z +\frac{1}{16}(  21 b_{20} z^2 - 6 b_{20} z w + b_{20} w^2),
\\
\dot{w}= &  -w + \frac{1}{16}(-27 b_{20} z^2  + 18 b_{20} z w - 7 b_{20} w^2).
\end{aligned}
\end{equation}
System \eqref{sys4-1} has the  Darboux factors
\begin{equation*}
\begin{aligned}
 l_1 =&   z +\frac{1}{16}3 b_{20} z^2 +\frac{1}{8}b_{20} z w + \frac{1}{48}b_{20} w^2,
 \\
  l_2 =&  w + \frac{1}{16}9 b_{20} z^2 + \frac{3 b_{20}}{8} z w + \frac{b_{20}}{16} w^2,
 \\
 l_3 =&   1 + 3 b_{20} z +  \frac{27 b_{20}^2 }{8}z^2 + b_{20} w - \frac{3 b_{20}^2}{4} z w +
  \frac{3 b_{20}^2}{8} w^2,
 \\
l_4 =&  1 +  \frac{3 b_{20}}{2} z +  \frac{9 b_{20}^2}{16} z^2  +  \frac{b_{20}}{2} w  +
  \frac{3 b_{20}^2}{8} z w +  \frac{b_{20}^2}{16} w^2,
 \end{aligned}
\end{equation*}
yielding  the Darboux linearization
\begin{equation*}
 z_1 =  l_1 l_3^{\alpha_1} l_4^{\alpha_2},
\hspace{0.3cm}
 w_1 =   l_2 l_3^{\beta_1} l_4^{\beta_2},
 \end{equation*}
 where $\alpha_1 = 1, ~ \alpha_2 =  -3,  ~ \beta_1 = 1$ and $\beta_2 =  -3$. \end{proof}
%

\bigskip


\section{Global dynamics of  system \eqref{sys-Ric3} having an isochronous center}
\label{Sec:global}

\medskip


Global phase portrait of a planar autonomous system is usually plotted on the Poincar\'e disc, which is obtained using
the Poincar\'e  compactification. We remind the procedure briefly, for more details see for instance  \cite{AndLGM1973, D-L-A}.

Consider  the planar vector field
$$
\mathcal{X}= \tilde{P}(x,y) \dfrac{\partial}{\partial x} + \tilde{Q}(x,y) \dfrac{\partial}{\partial y},
$$
where $\tilde{P}(x,y)$ and $\tilde{Q}(x,y)$ are polynomials of degree $n$.
Let  $\mathbb{S}^2 = \{ y=(y_1,y_2,y_3) \in  \mathbb{R}^3 : y_1^2+y_2^2+y_3^2 =1 \}$,
$\mathbb{S}^1$ be the equator of  $\mathbb{S}^2$ and $p(\mathcal{X})$ be the \textit{Poincar\'e compactification} of $\mathcal{X}$ on $\mathbb{S}^2$.
On  $\mathbb{S}^2 \setminus \mathbb{S}^1$ there are two symmetric copies of $\mathcal{X}$, and once we know
the behaviour of $p(\mathcal{X})$ near $\mathbb{S}^1$, we know the behaviour of $\mathcal{X}$ in a neighbourhood of the infinity.
The Poincar\'e compactification has the property that $\mathbb{S}^1$ is invariant under
the flow of $p(\mathcal{X})$. The projection of the closed northern hemisphere of $\mathbb{S}^2$ on $y_3 = 0$
under $(y_1, y_2, y_3) \mapsto (y_1, y_2)$ is called the \textit{Poincar\'e disc}, and its boundary is $\mathbb{S}^1$.

Because $\mathbb{S}^2$ is a differentiable manifold, we consider the six local charts $U_i = \{ y \in  \mathbb{S}^2 : y_i > 0 \}$ and $V_i = \{ y \in  \mathbb{S}^2 : y_i < 0 \}$ for computing the expression of $p(\mathcal{X})$ where $i = 1, 2, 3$.
The diffeomorphisms $F_i : U_i \rightarrow \mathbb{R}^2$ and $G_i : V_i \rightarrow \mathbb{R}^2$ for $i = 1, 2, 3$ are the inverses of the central projections from the planes tangent at the points $(1, 0, 0), (-1, 0, 0), (0, 1, 0), (0,-1, 0), (0, 0, 1)$, and $(0, 0,-1)$ respectively.
We denote by $(u, v)$ the value of $F_i(y)$ or $G_i(y)$ for any $i = 1, 2, 3$.

The expression for $p(\mathcal{X})$ in the local chart $(U_1, F_1)$ is given by
$$
\dot{u} = v^n \left[ -u \tilde{P} \left( \dfrac{1}{v}, \dfrac{u}{v}\right) + \tilde{Q} \left( \dfrac{1}{v}, \dfrac{u}{v}\right) \right],
\hspace{0.3cm}
\dot{v} = - v^{n+1} \tilde{P} \left( \dfrac{1}{v}, \dfrac{u}{v}\right),
$$
for $(U_2, F_2)$ is
$$
\dot{u} = v^n \left[ \tilde{P} \left( \dfrac{u}{v}, \dfrac{1}{v}\right) -u \tilde{Q} \left( \dfrac{u}{v}, \dfrac{1}{v}\right) \right],
\hspace{0.3cm}
\dot{v} = - v^{n+1} \tilde{Q} \left( \dfrac{u}{v}, \dfrac{1}{v}\right),
$$
and for $(U_3, F_3)$ is
$$
\dot{u} = \tilde{P}(u,v),
\hspace{0.3cm}
\dot{v} = \tilde{Q}(u,v).
$$

The expressions for $V_i$'s are the same as that
for $U_i$'s but multiplied by the factor $(-1)^{n-1}$. In these coordinates $v = 0$ always denotes the points of $\mathbb{S}^1$.
When we study the infinite singular points on the charts $U_2 \cup V_2$, we only need to verify if the origin of these charts are
singular points.

It is said that two polynomial vector fields $\mathcal{X}$ and $\mathcal{Y}$ on $\mathbb{R}^2$ are \textit{topologically equivalent} if there exists a homeomorphism on
$\mathbb{S}^2$ preserving the infinity $\mathbb{S}^1$ carrying orbits of the flow induced by $p(\mathcal{X})$ into orbits of the flow induced by $p(\mathcal{Y})$, preserving
or not the sense of all orbits.


\medskip

In this section, we study the global structures of system \eqref{sys-Ric3} in Poincar\'e discs for the case when it has an isochronous center  listed in Theorem \ref{Theorem 1}.

\begin{theorem}
\label{th-global}
The global phase portrait of system \eqref{sys-Ric3} possessing an isochronous center
 listed in Theorem \ref{Theorem 1} is topologically equivalent to one of
 phase portraits in Fig. \ref{fig-theo}.
 More precisely,  there exists only one equilibrium of system  \eqref{sys-Ric3} in the plane, which is an isochronous center at the origin.
  The neighborhood of equilibrium at infinity  consists of one elliptic sector and three hyperbolic sectors (or
    two  hyperbolic sectors and two parabolic sectors)  under conditions $(2)$ and $b_{11}\ne0$ (or under conditions $(3)$ and $b_{02}\ne0$);
    otherwise, the isochronous center is global.
\begin{figure}[htbp]
  \centering
\includegraphics[scale=0.35]{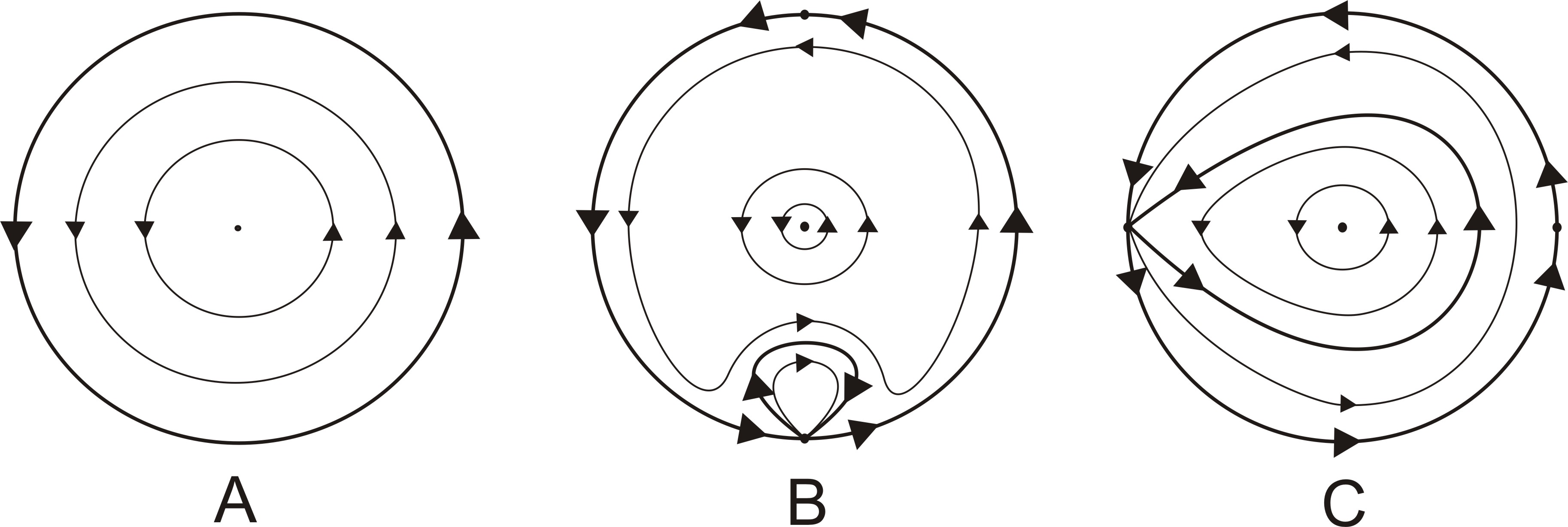}
  \caption{Global phase portraits of system \eqref{sys-Ric3} possessing an isochronous center  listed in Theorem \ref{Theorem 1}.}  \label{fig-theo}
  \end{figure}

\end{theorem}

\begin{proof}
From Theorem \ref{Theorem 1} we have that under conditions $(1)$--$(4)$ system \eqref{sys-Ric3} is linearizable.
Under conditions $(1)$ and $(4)$ real systems   \eqref{sys-Ric3} becomes the linear system $\dot{x}=-y, ~\dot{y}=x$ and
its phase portrait is presented in Figure \ref{fig-theo}.A.

Under conditions $(2)$ and $(3)$ system  \eqref{sys-Ric3} becomes
\begin{equation} \label{sys2-2}
\dot{x} =     -y,
\qquad
\dot{y} =  x +   b_{11} x y  +   \frac{b_{11}^2}{9} x^3,
 \end{equation}
and
\begin{equation} \label{sys2-3}
\dot{x} =     -y -  \frac{4}{9} b_{02}^2 y^3,
\qquad
\dot{y} =  x +   b_{02} y^2,
 \end{equation}
respectively.

Note that if $b_{11}=0$ in \eqref{sys2-2} and $b_{02}=0$ in \eqref{sys2-3}, then both systems are the canonic linear systems
and have a global center shown in Figure \ref{fig-theo}.A.
Thus,  we consider the cases when $b_{11} \neq 0$ and $b_{02} \neq 0$. In both cases by a linear change of coordinates we can reduce systems \eqref{sys2-2} and \eqref{sys2-3} to systems
\begin{equation} \label{sys2-2without}
\dot{x} =     -y,
\qquad
\dot{y} =  x +  x y  +  \frac{x^3}{9},
 \end{equation}
and
\begin{equation} \label{sys2-3without}
\dot{x} =     -y -  \frac{4}{9} y^3,
\qquad
\dot{y} =  x +   y^2,
 \end{equation}
respectively.

System  \eqref{sys2-2without} has only the isochronous center at $(0,0)$ as a finite singular point.
Now we analyze its singular points at infinity.
In the local chart $U_1$ system \eqref{sys2-2without} becomes
$$
\dot{u} =  \frac{1}{9} (1 + 9 u v + 9 v^2 + 9 u^2 v^2),
\qquad
\dot{v} = u v^3.
$$
This system has no real singular points.
So the unique possible infinite singular point is the origin of the local chart $U_2$.
In the local chart $U_2$ system \eqref{sys2-2without} becomes
\begin{equation} \label{sys2-2-U2}
\begin{aligned}
\dot{u} =  \frac{1}{9} (- u^4 - 9 u^2 v - 9 v^2 - 9 u^2 v^2),
\qquad
\dot{v} = - \frac{1}{9} u v (u^2 + 9 v + 9 v^2).
\end{aligned}
\end{equation}
It is  clear that $(0,0)$ is a singular point of \eqref{sys2-2-U2} and the linear part of \eqref{sys2-2-U2} at $(0,0)$
is the null matrix, i.e, $ \left(
\begin{array}{cc}
0 & 0 \\
  0 & 0
  \end{array}
\right). $
Applying the directional blow-up  in the $v$-axis twice we obtain that the behaviour of the orbits close to the origin of $U_2$ is as in Figure \ref{fig-1}.
Therefore,
the global phase portrait of system \eqref{sys2-2without} is topologically equivalent to the one in Figure \ref{fig-theo}.B.

\begin{figure}[htbp]
  \centering
  \begin{minipage}[b]{0.43\textwidth}
    \includegraphics[scale=.54]{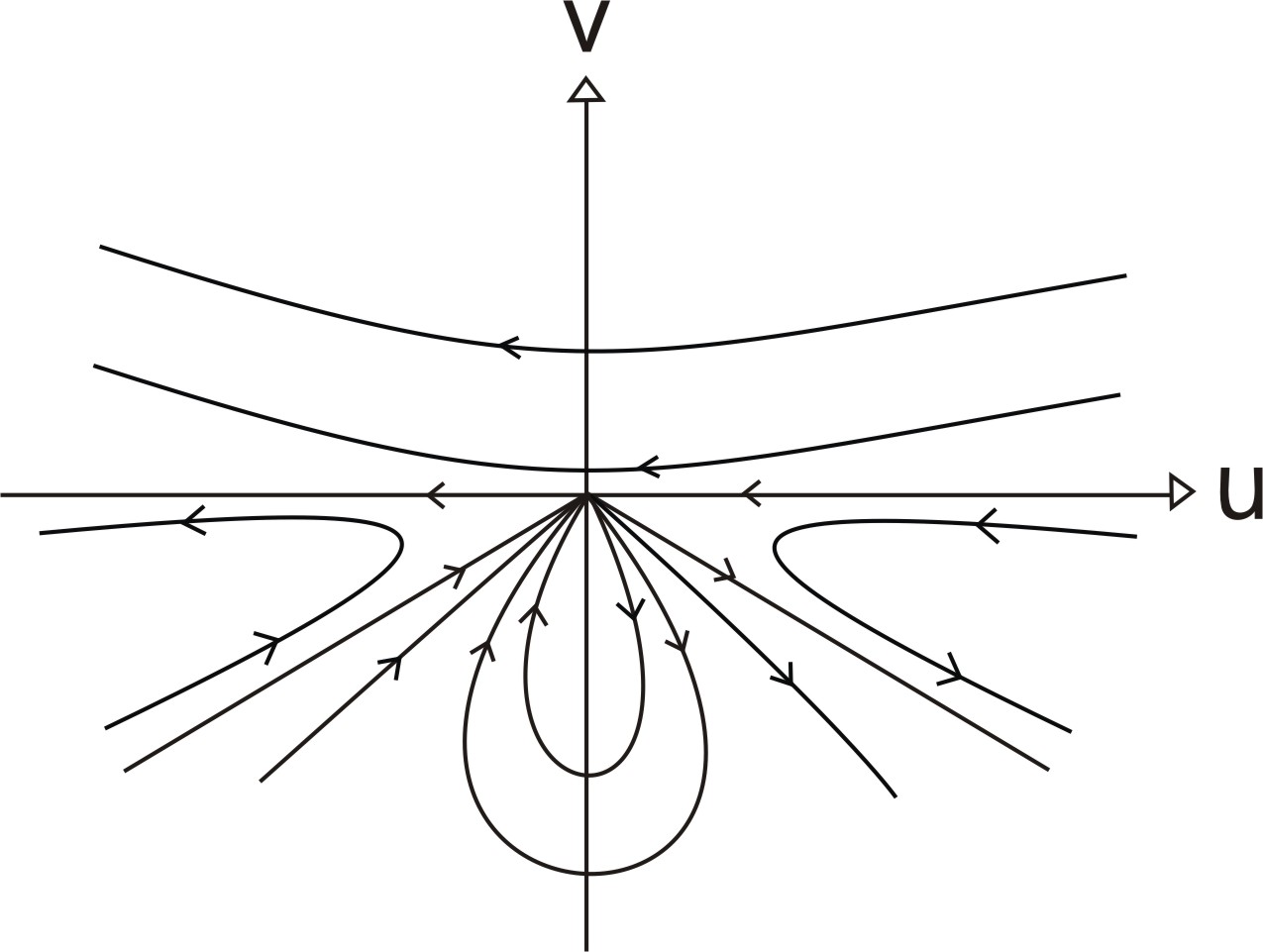}
    \caption{Behaviour of the orbits  close to the origin of system \eqref{sys2-2-U2}.} \label{fig-1}
  \end{minipage}
  \hfill
  \begin{minipage}[b]{0.43\textwidth}
    \includegraphics[scale=.54]{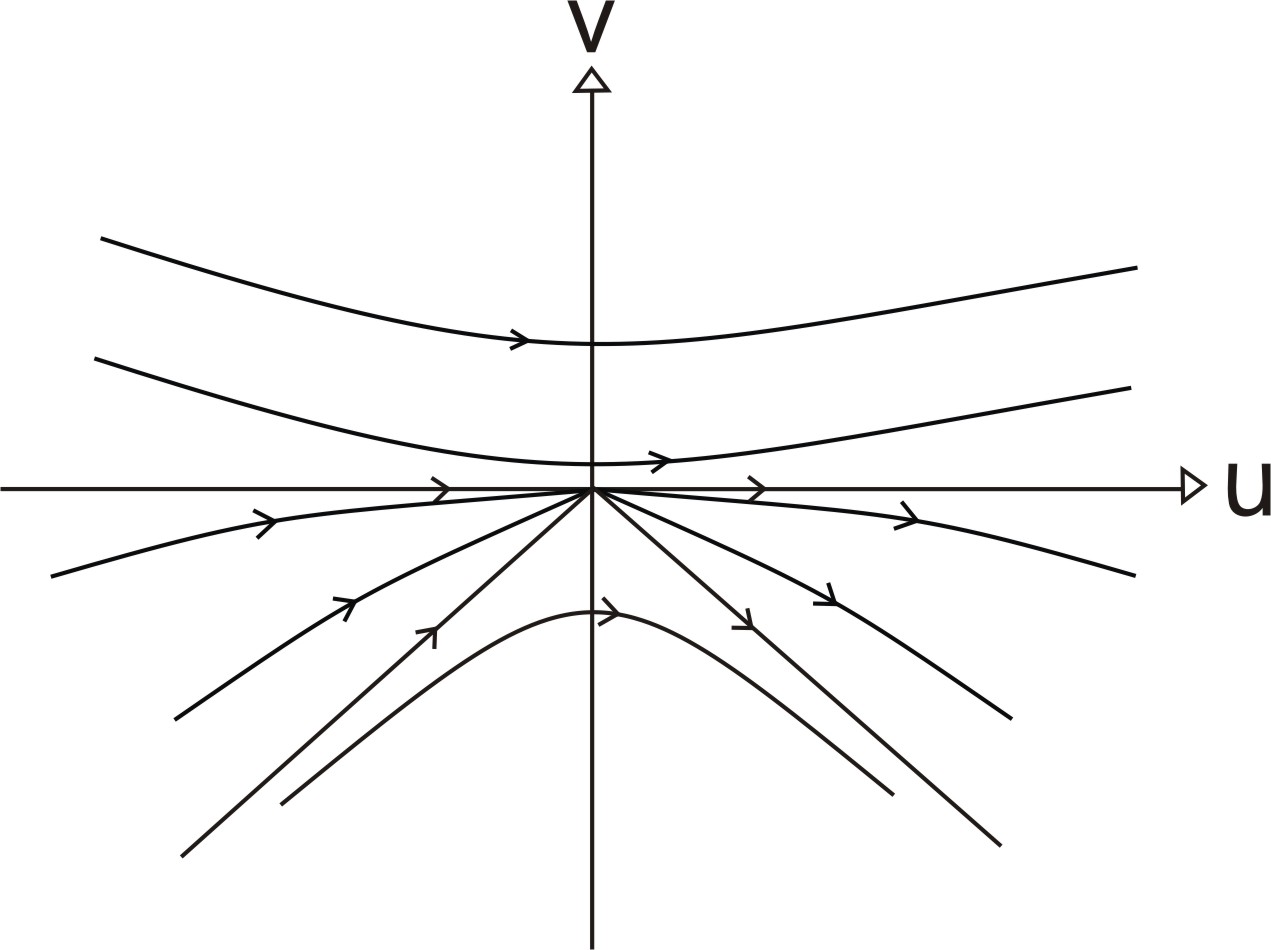}
    \caption{Behaviour of the orbits  close to the origin of system \eqref{sys2-3-U1}.} \label{fig-2}
  \end{minipage}
\end{figure}

Now we study system  \eqref{sys2-3without}.
This system has only the isochronous center at $(0,0)$ as a finite singular point.
For the infinite singular points,
in the local chart $U_1$ system \eqref{sys2-3without} becomes
\begin{equation} \label{sys2-3-U1}
\begin{aligned}
\dot{u} =  \frac{1}{9} (4 u^4 + 9 u^2 v + 9 v^2 + 9 u^2 v^2),
\qquad
\dot{v} = \frac{1}{9} u v (4 u^2 + 9 v^2) .
\end{aligned}
\end{equation}
This system has only $(0,0)$ as a singular point, and the linear part of \eqref{sys2-3-U1} at $(0,0)$ is the null matrix.
Applying  the directional blow-up  in the $v$-axis twice we obtain that the behaviour of the orbits close to the origin of $U_1$ is as showing in Figure \ref{fig-2}.

In the local chart $U_2$ system \eqref{sys2-3without} becomes
\begin{equation} \label{sys2-3-U2}
\begin{aligned}
\dot{u} =  \frac{1}{9} (-4 - 9 u v - 9 v^2 - 9 u^2 v^2),
\qquad
\dot{v} = - v^2 (1 + u v).
\end{aligned}
\end{equation}
As it is mentioned above, we need to study only the origin of this chart, but $(0,0$) is not a singular point for system  \eqref{sys2-3-U2}.
Thus, 
 the global phase portrait of system \eqref{sys2-3without} is topologically equivalent to the portrait in Figure \ref{fig-theo}.C.
\end{proof}


\bigskip

\section{Weak center and local bifurcation of critical periods}

\medskip


Let $\alpha=(a_{20}, a_{11}, ..., b_{20}, b_{11},...)$ be the string of parameters of real system \eqref{System general real} with a center at the origin.
Changing the system to the polar coordinates $x = r\cos\theta$, $y = r \sin\theta$ and eliminating $t$, we obtain
\begin{equation} \label{dr-theta}
 \frac{dr}{d\theta}=r \frac{x \dot{x}+y \dot{y}}{x \dot{y}-y \dot{x}}=\frac{r H(r, \theta,\alpha) }{1+G(r, \theta,\alpha) },
\end{equation}
where $H(r, \theta,\alpha)$ and $G(r, \theta,\alpha)$ are polynomials of $r, \alpha,\cos\theta$ and $\sin\theta$.
The solution $r= r(\theta, \alpha)$ of equation \eqref{dr-theta} satisfying the initial condition  $r(0, \alpha) = r_0 > 0$ may
be locally represented as a convergent power series in $r_0$,
\begin{equation} \label{r-r0}
 r(\theta, \alpha) =\displaystyle\sum_{k=1}^{\infty} ~v_k(\theta, \alpha) r_0^k.
\end{equation}
Substituting \eqref{r-r0} into   \eqref{dr-theta}, one can find coefficients $v_k(\theta, \alpha)$ $(k> 1)$ by successive
integration.

Assuming that  $\Gamma_{r_0}$ is the closed trajectory through $(r_0, 0)$, we can compute the period function  as
\begin{equation*}
\begin{aligned}
T(r_0, \alpha)&=\oint_{\Gamma_{r_0}} ~dt=\int_0^{2\pi} ~\frac{d\theta}{1+G(r, \theta,\alpha)}
= \displaystyle\sum_{k=0}^{\infty} ~p_{k}(\alpha) r_0^k.
\end{aligned}
\end{equation*}


The period function is even  and has the Taylor series expansion	
\begin{equation} \label{P-func}
\begin{aligned}
T(r_0, \alpha)=2\pi + \displaystyle\sum_{k=1}^{\infty} ~p_{2k}(\alpha) r_0^{2k},
\end{aligned}
\end{equation}
where $r_0<\delta$ and coefficients $p_{2k}$'s  are polynomials in parameters of system \eqref{System general real}
(see e.g. \cite{AmLS1982, CJ1989, M-R-T, R-S}).

%

 If $ p_2=... =p_{2k}=0$ and $p_{2k+2}\ne 0$, then the origin of system \eqref{System general real} is a {\it weak
 center} of  order $k$. If $p_{2k}=0$ for each $k\ge 1$, then the origin  is an {\it isochronous center}.
For a center which is not isochronous, a {\it local critical period} is any value $\tilde{r}_0<\delta$ for which
$T'(\tilde{r}_0)=0$.



By classical results of local critical period bifurcations \cite{CJ1989}, at most $k$ local critical periods
can bifurcate from the period function related to a weak center of order $k$. In order to prove that there are perturbations with exactly
$k$ local critical periods, we remind  Theorem 2    of \cite{YuH2009} as follows.


\begin{theorem}
Assume that the period constants $p_{2j}$ ($j=1,2,...,k$) of system \eqref{System general real} depend
on $k$ independent parameters $a_1, a_2, ..., a_k$. Suppose that there exists $\tilde{a}= (\tilde{a}_1, \tilde{a}_2, ..., \tilde{a}_k)$ such
that
\begin{equation*}
\begin{aligned}
& p_{2j}(\tilde{a})  =p_{2j}(\tilde{a}_1, \tilde{a}_2, ..., \tilde{a}_k)=0, ~~j=1,2,...,k,
\\
& p_{2k+2}(\tilde{a})\ne 0
\end{aligned}
\end{equation*}
 and
 $$
\det\Big( \frac{\partial ( p_2, p_4,...,p_{2k}) }{\partial (a_1, a_2, ... , a_k) } (\tilde{a})  \Big)\ne 0,
$$
then $k$  critical periods bifurcate from the center at the origin of system \eqref{System general real} after small appropriate  perturbations.
\label{th-crit-bif}
\end{theorem}

{\it Remark}. The proof that $k$ critical periods can bifurcate after perturbations
of system \eqref{System general real} corresponding
to parameters $\tilde{a}$ is derived using the Implicit Function Theorem,
and the proof that  the bound   $k$ is sharp   can be  derived either  using
the Mean Values Theorem \cite{Han} or  Rolle's Theorem \cite{Bau}.
In practice   $k$ critical periods
can be obtained choosing perturbations such that for some system  $a^*$ close to
$$
|p_2(a^*) |\ll | p_4(a^*)|\ll \dots \ll |p_{2k}(a^*)|\ll |p_{2k+2}(a^*)|$$
and the signs in the sequence $p_2(a^*),   p_4(a^*), \dots  p_{2k}(a^*),  p_{2k+2}(a^*)$ alternate (see e.g. \cite{Hanb,LHR,R-S} for more details).



\smallskip


Because bifurcations of critical periods are bifurcations from centers,  to study them for  system \eqref{sys-Ric3} we need to know
the center variety of the system. Due to computational difficulties the center variety of  system \eqref{sys-Ric3}
has been found only in the case when  $a_{03}=0$ \cite{ZRY}. So, from now on we assume that  in  system \eqref{sys-Ric3}
$a_{03}=0$ and consider the system
\begin{equation} \label{sys-a03}
\begin{aligned}
\dot{x}& =     -y + a_{02} y^2,
\\
\dot{y} &=  x + (b_{20} x^2 + b_{11} x y + b_{02} y^2) +  (b_{30} x^3 + b_{21} x^2 y +
    b_{12} x y^2).
  \end{aligned}
\end{equation}
The centers of system \eqref{sys-a03} are identified in the following theorem.

 \begin{theorem}[\cite{ZRY}]  \label{th-center}
 System \eqref{sys-a03} has a center at the origin if the 7-tuple
 of its parameters belongs to the variety of one of the following prime
ideals:
 \begin{enumerate}
\item $I_1= \langle b_{21}, b_{20}, b_{02}  \rangle$,
\item $I_2= \langle  b_{30},    b_{12},    b_{02},   b_{11} b_{20}-b_{21} \rangle$,
\item $I_3= \langle  b_{30}, b_{21}, b_{12}, -2 b_{02} b_{11}^2 + 4 b_{02}^2 b_{20} - b_{11}^2 b_{20},
 2 a_{02} b_{11} + b_{11}^2 - 4 b_{02} b_{20}, 2 a_{02} b_{02} - b_{02} b_{11} - b_{11} b_{20},
 4 a_{02}^2 - b_{11}^2 - 4 b_{20}^2 \rangle$,
\item $I_4= \langle  b_{21}, b_{11},a_{02} \rangle$,
\item $I_5= \langle  a_{02}, b_{02} b_{21} + b_{11} b_{30}, 2 b_{02} b_{12} + b_{12} b_{20} + b_{02} b_{30}, b_{02} b_{11} + b_{11} b_{20} - b_{21},
 b_{02}^2 + b_{02} b_{20} + b_{30}, b_{12} b_{20} b_{21} - 2 b_{11} b_{12} b_{30} - b_{11} b_{30}^2,
 b_{11} b_{20} b_{21} - b_{21}^2 - b_{11}^2 b_{30}, b_{12} b_{20}^2 - 4 b_{12} b_{30} - b_{02} b_{20} b_{30} - 2 b_{30}^2,
 b_{11} b_{12} b_{20} - 2 b_{12} b_{21} + b_{11} b_{20} b_{30} - b_{21} b_{30},
 -(b_{12} b_{21}^2) + b_{11}^2 b_{12} b_{30} + b_{11}^2 b_{30}^2 \rangle$,
\item $I_6= \langle  b_{21},b_{12},b_{11},b_{02} \rangle$,
\item $I_7= \langle  b_{21}, b_{12},b_{30}, 3 b_{02}+5 b_{20},5 a_{02}-b_{11}, 6 b_{11}^2+25 b_{20}^2 \rangle$.
\end{enumerate}
\end{theorem}

{\it Remark.} Like in the proof of our Theorem \ref{Theorem 1}
modular computations were used in order to determine centers of system \eqref{sys-a03},
so it can happen that the list of centers of the  system given in Theorem \ref{th-center} is incomplete. For this reason it stands in the theorem "if" but not "if and only if".

We consider the local bifurcations of critical periods  for  system \eqref{sys-a03} when all parameters are real.
Because in $\mathbb{R}^7$ the variety of $I_7$ consists of one point which is the origin $(0,0,0,0,0,0,0)\in \mathbb{R}^7$, we only need to consider varieties of first six  ideals $I_1-I_6$.


 \begin{theorem}
 Suppose that the origin $O: (0,0)$  of system \eqref{sys-a03} is a weak center of a finite order.
 \\
(1) Then the order is at most $3$. More precisely,
the order is  at most $3$ (resp. $ 0, 0, 3, 2, 2$) when  parameters belong to the variety of the ideal
$I_1$ (resp. $I_2-I_6$).
\\
(2) Moreover, at most $3$ (resp. $ 0, 0, 3, 2, 2$)  critical periods can be bifurcated from the weak center $O$ of system
 \eqref{sys-a03} and there exists a perturbation with exactly $3$ (resp. $ 0, 0, 3, 2, 2$) critical periods bifurcated from $O$
 when  parameters belong to the variety of the ideal  $I_1$ (resp. $I_2-I_6$).
\label{Th-weak-center}
\end{theorem}

\begin{proof}
When the parameter  $\alpha=(a_{02}, b_{20}, b_{11}, b_{12}, b_{02}, b_{21}, b_{30})$ belongs to the variety of the ideal  $I_1$,
we found that the first four period coefficients of \eqref{P-func}  are 
\begin{eqnarray*}
p_{1,2}(\alpha)&=&  10 a_{02}^2-a_{02} b_{11}+b_{11}^2-3 b_{12}-9 b_{30},
\\
p_{1,4}(\alpha)&=&
1540 a_{02}^4+700 a_{02}^3 b_{11}+21 a_{02}^2 b_{11}^2-2 a_{02} b_{11}^3+b_{11}^4+84 a_{02}^2 b_{12}+300 a_{02}^2 b_{30}
\\
&& +6 a_{02} b_{11} b_{12}+18 a_{02} b_{11} b_{30}-6 b_{11}^2 b_{12}-66 b_{11}^2 b_{30}+9 b_{12}^2+54 b_{12} b_{30}+513 b_{30}^2,
\\
p_{1,6}(\alpha)&=&  3403400 a_{02}^6+3303300 a_{02}^5 b_{11}+690690 a_{02}^4 b_{11}^2-699 a_{02}^2 b_{11}^4+281340 a_{02}^3 b_{11} b_{30}
\\
&&-139 b_{11}^6+1261260 a_{02}^4 b_{12}+1455300 a_{02}^4 b_{30}+346500 a_{02}^3 b_{11} b_{12}+11935 a_{02}^3 b_{11}^3
\\
&&+7263 a_{02}^2 b_{11}^2 b_{12}-4995 a_{02}^2 b_{11}^2 b_{30}-3366 a_{02} b_{11}^3 b_{12}-2538 a_{02} b_{11}^3 b_{30}+417 a_{02} b_{11}^5
\\
&&+1251 b_{11}^4 b_{12}+1377 b_{11}^4 b_{30}+6966 a_{02}^2 b_{12}^2+31860 a_{02}^2 b_{12} b_{30}-52650 a_{02}^2 b_{30}^2
\\
&& +13878 a_{02} b_{11} b_{12} b_{30}-4455 a_{02} b_{11} b_{30}^2-3321 b_{11}^2 b_{12}^2 +2025 b_{12}^3+5265 b_{12}^2 b_{30}
\\
&&
-7614 b_{11}^2 b_{12} b_{30}+41391 b_{11}^2 b_{30}^2-13365 b_{12} b_{30}^2-382725 b_{30}^3+7209 a_{02} b_{11} b_{12}^2.
\end{eqnarray*}
We omit the expression of   $p_{1,8}(\alpha)$,  since it is long and the number of its terms  is   $55$.

  We compute the decomposition of $\langle p_{1,2}, p_{1,4}, p_{1,6}, p_{1,8} \rangle  $ with \texttt{minAssGTZ}
and obtain $\langle a_{02},  b_{11}^2-9 b_{30},  b_{12}  \rangle  $.
That is, the condition $p_{1,2}= p_{1,4}=p_{1,6}= p_{1,8}=0$ yields that $b_{12} = a_{02} =  b_{20} =  b_{02} =  b_{21} =  a_{03} =  9 b_{30} - b_{11}^2 =0 $,
showing that the origin is an isochronous center of system \eqref{sys-a03} in this case by Theorem \ref{Theorem 1}.

Solving the equation $p_{1,2}(\alpha)=0$ we get
\begin{eqnarray}
b_{12}= \tilde{b}_{12}:= (10/3) a_{02}^2-(1/3) a_{02} b_{11}+(1/3) b_{11}^2-3 b_{30}.
\label{b12-}
\end{eqnarray}
Substituting \eqref{b12-} in $p_{1,4}(\alpha)$, we obtain
\[
 432 b_{30}^2 +48( a_{02}^2 -b_{11}^2) b_{30} +1920 a_{02}^4+672 a_{02}^3 b_{11}+48 a_{02}^2 b_{11}^2=0.
\]
Thus, when $-1439 a_{02}^4-504 a_{02}^3 b_{11}-38 a_{02}^2 b_{11}^2+b_{11}^4<0$ the origin $O$
is a weak center of order $1$. When $-1439 a_{02}^4-504 a_{02}^3 b_{11}-38 a_{02}^2 b_{11}^2+b_{11}^4\ge 0$, from  $p_{1,4}(\alpha)=0$ we find that
 \begin{eqnarray*}
 b_{30} =  \tilde{b}_{30} :=\frac{1}{18} \Big(- a_{02}^2+  b_{11}^2+ \sqrt{-1439 a_{02}^4-504 a_{02}^3 b_{11}-38 a_{02}^2 b_{11}^2+b_{11}^4} \Big).
\end{eqnarray*}

We now employ the procedure {\it Reduce} of  computer algebra system {\sc Mathematica}   for the set of  equalities and inequalities
$\{b_{12}= \tilde{b}_{12},  b_{30} =  \tilde{b}_{30},  -1439 a_{02}^4-504 a_{02}^3 b_{11}-38 a_{02}^2 b_{11}^2+b_{11}^4\ge 0, p_{1,6}(\alpha)=0, p_{1,8}(\alpha)\ne0 \}$,
and find that this semi-algebraic system is fulfilled if and only if $a_{02}\ne0$ and
\begin{eqnarray}
&& -128966505300 a_{02}^{10} - 131928442900 a_{02}^9 b_{11} -62892021225 a_{02}^8 b_{11}^2 - 18497447700 a_{02}^7 b_{11}^3
\nonumber \\
&&~ -3614043210 a_{02}^6 b_{11}^4 - 467726370 a_{02}^5 b_{11}^5 -37088580 a_{02}^4 b_{11}^6 - 1472760 a_{02}^3 b_{11}^7
\nonumber \\
&& ~ - 19938 a_{02}^2 b_{11}^8 +1650 a_{02} b_{11}^9 + 125 b_{11}^{10} =0.
\label{I-p6}
\end{eqnarray}
Assuming that  $a_{02} = 1/2$, we can calculate one of solutions $b_{11}\approx -2.405222225$ 
from  above equation, which indicates the existence of solutions of above equation with respect to parameters
$a_{02}$ and $b_{11}$ in real field.
Moreover, computing with {\sc Mathematica}  the rank of the matrix
$$
\frac{\partial(p_{1,2},p_{1,4},p_{1,6})}{\partial(a_{02},b_{11},b_{12},b_{30})},
$$
we find that it is equal to $3$ when $b_{12}= \tilde{b}_{12},  b_{30} =  \tilde{b}_{30}$, $a_{02}\ne0$ and \eqref{I-p6} holds.
From  Theorem  \ref{th-crit-bif}
there exists a perturbation of system  \eqref{sys-a03} with exactly  $3$ critical periods bifurcated from weak center $O$
of order $3$ when  $\alpha$ belongs to the variety of  $I_1$.



\smallskip


When the parameter  $\alpha=(a_{02}, b_{20}, b_{11}, b_{12}, b_{02}, b_{21}, b_{30})$ belongs to the variety of the ideal  $I_2$,
we have the first  period coefficient  in \eqref{P-func}:
\begin{eqnarray*}
p_{2,2}(\alpha)&=&  10 a_{02}^2-a_{02} b_{11}+b_{11}^2+10 b_{20}^2,
\end{eqnarray*}
which cannot be equal to zero in the real field, since $10 a_{02}^2-a_{02} b_{11}+b_{11}^2>0$ unless $a_{02}=b_{11}=0$.
That is, the center at the origin is of order $0$ in this case.


When the parameter  $\alpha=(a_{02}, b_{20}, b_{11}, b_{12}, b_{02}, b_{21}, b_{30})$ lies in the variety of the ideal  $I_3$,
we compute the first  period coefficient  in \eqref{P-func}:
\begin{eqnarray*}
p_{3,2}(\alpha)&=&   10 a_{02}^2-a_{02} b_{11}+b_{11}^2  +4 b_{02}^2+10 b_{02} b_{20}+10 b_{20}^2,
\end{eqnarray*}
finding that $p_{3,2}(\alpha)\ne 0$  unless all parameters vanish. Thus the  center $O$ is of order $0$ in this case.


When the parameter  $\alpha=(a_{02}, b_{20}, b_{11}, b_{12}, b_{02}, b_{21}, b_{30})$ belongs to the variety of the ideal  $I_4$ or $I_5$,
we can see that system \eqref{sys-a03} is a reduced Kukles system.
The variety of ideal  $I_4$  (resp. $I_5$) for  center conditions    corresponds  to the center type $K_{III}$ (resp. $K_{II}$ or $K_{IV}$) in \cite{RT1997}.
Applying Theorems 3.3, 3.4 and 3.7 of  \cite{RT1997}  we obtain that
 the order at the  origin is at most $3$ (resp.  $2$), and there exists a perturbation with exactly  $3$ (or $ 2$) critical periods bifurcated from $O$
 when  parameters belong to the variety of the ideal $I_4$ (resp. $I_5$).


When the parameter  $\alpha=(a_{02}, b_{20}, b_{11}, b_{12}, b_{02}, b_{21}, b_{30})$ belongs to the variety of the ideal  $I_6$,
we found that the first three period coefficients in \eqref{P-func} are
\begin{eqnarray*}
p_{6,2}(\alpha)&=&   10 a_{02}^2+10 b_{20}^2-9 b_{30},
\\
p_{6,4}(\alpha)&=& 1540 a_{02}^4+200 a_{02}^2 b_{20}^2+1540 b_{20}^4+300 a_{02}^2 b_{30}-3300 b_{20}^2 b_{30}+513 b_{30}^2,
\\
p_{6,6}(\alpha)&=&  136136 a_{02}^6+38808 a_{02}^4 b_{20}^2+13080 a_{02}^2 b_{20}^4+165704 b_{20}^6+58212 a_{02}^4 b_{30}
\\
&& ~-17496 a_{02}^2 b_{20}^2 b_{30}-546588 b_{20}^4 b_{30}-2106 a_{02}^2 b_{30}^2+341334 b_{20}^2 b_{30}^2-15309 b_{30}^3.
\end{eqnarray*}
Eliminating $b_{30}$ from $p_{6,2}(\alpha)=0$ we find
\[
b_{30}=\hat{b}_{30}:=(10 a_{02}^2+10 b_{20}^2)/9.
\]
Letting $b_{30}=\hat{b}_{30}$ we obtain from  $p_{6,4}=0$ that
$$
 47 a_{02}^4-35 a_{02}^2 b_{20}^2-28 b_{20}^4=0,
$$
yielding
\[
a_{02} = \hat{a}_{02} :=\pm \sqrt{35/94+(3/94) \sqrt{721}} ~b_{20}.
\]
Eliminating  $b_{30}$ and $a_{02}$ by substituting $b_{30}=\hat{b}_{30}$ and $a_{02} = \hat{a}_{02}$ into $p_{6,6}(\alpha)$,
we obtain
\[
 b_{20}^6 (3578681+142373 \sqrt{721}),
\]
which does not vanish if $b_{20}\ne 0$.
Therefore, the order of the weak center is at most  $2$, and there exists a perturbation with exactly    $ 2$  critical periods bifurcated from $O$
 when  parameters belong to the variety of the ideal $I_6$ by Theorem  \ref{th-crit-bif},
  since  the rank of the matrix
$$
\frac{\partial(p_{6,2},p_{6,4})}{\partial(a_{02},b_{20},b_{30})},
$$
 is equal to $2$ when $b_{30}=\hat{b}_{30}$, $a_{02} = \hat{a}_{02}$ and $b_{20}\ne 0$.
Notice that when $b_{20}=0$ and $a_{02}\ne 0$  the center $O$
is a weak center of order $1$, and when $b_{20}=a_{02}=0$ the center $O$
is either the linear  isochronous center  or the order is $0$.
\end{proof}

\medskip

\section{Conclusion}

For cubic generalized Riccati system  \eqref{sys-Ric3}, we derived conditions on parameters of the system for the linearizability  of the origin,
see conditions $(1)$-$(4)$ of Theorem \ref{Theorem 1}.

For the study we have used the approach based on the modular
calculations of the set of solutions of polynomial systems,
which was used for the first time in \cite{R-C-H} and described in details
in \cite{R-P}.  The approach can be considered as  one between precise
symbolic computations and numerical computations
since it produces a result which is not completely  correct, but
correct with high probability -- in the sense that it is easily verified
if the obtained solutions of a given system of polynomials
are correct, but it can happen, that some solutions are lost.
Recently an efficient algorithm to verify if the list of solutions
obtained with the approach is complete was proposed
in  \cite{NY} however it is not yet implemented in freely
available computer algebra systems. The approach can
be efficiently applied to study  various mathematical  models
where arises the problem of solving polynomial equations.

When the origin is an isochronous center, we found that system \eqref{sys-Ric3} has at most three topologically equivalent global  structures,
which are the global center at the origin, the neighborhood of equilibrium at infinity  consists of one elliptic sector and three hyperbolic sectors, and
 the neighborhood of equilibrium at infinity  consists of two hyperbolic sectors and two parabolic sectors,  as shown in Theorem \ref{th-global}.
The last result is the investigation of local bifurcations  of critical periods in
a neighborhood of the center. We proved that the order of  weak center at the origin  is at most $3$   when  parameters belong to the center variety
and at most $3$   critical periods can be bifurcated from the weak center  of system  \eqref{sys-a03},  as shown in Theorem \ref{Th-weak-center}.


\section*{Acknowledgements}

The first author acknowledges  the financial support from the Slovenian Research Agency
(research core funding No. P1-0306).
The second author is partially supported by a CAPES grant.
The third author has received funding from the European Union's Horizon 2020 research and innovation
programme under the Marie Sklodowska-Curie grant agreement No 655212, and is  partially supported by the National Natural
Science Foundation of China (No. 11431008) and the RFDP of Higher Education of China grant (No. 20130073110074).
The first, second and third authors  are also   supported by Marie Curie International Research Staff Exchange Scheme Fellowship
within the 7th European Community Framework Programme, FP7-PEOPLE-2012-IRSES-316338.
The forth author is partially supported by  the National Natural Science Foundation of China (No. 11501370).
The first author thanks Professor Maoan Han for fruitful discussions on the work.



\section*{Appendix}

Here are listed the first two pairs of the linearizability quantities of system \eqref{sys-Ric3}.

{\footnotesize
\begin{equation*}
\begin{aligned}
i_1 = &  10 a_{02}^2 + 9 a_{03} + 4 b_{02}^2 - a_{02} b_{11} + b_{11}^2 - 3 b_{12} + 10 b_{02} b_{20} +
     10 b_{20}^2 - 9 b_{30},
   \\
   j_1 = & 2 a_{02} b_{02} - b_{02} b_{11} - b_{11} b_{20} + b_{21},
\\
i_2 = & 168 a_{02}^2 a_{03} - 272 a_{02}^2 b_{02}^2 - 72 a_{03} b_{02}^2 - 32 b_{02}^4 -
     112 a_{02}^3 b_{11} - 42 a_{02} a_{03} b_{11} + 40 a_{02} b_{02}^2 b_{11} -  21 a_{03} b_{11}^2
     \\
&      - 18 b_{02}^2 b_{11}^2 - 21 a_{02} b_{11}^3 - b_{11}^4 + 12 a_{02}^2 b_{12} -
     48 b_{02}^2 b_{12} + 72 a_{02} b_{11} b_{12} - 3 b_{11}^2 b_{12} + 18 b_{12}^2 -
     48 a_{02}^2 b_{02} b_{20}
     \\
     & - 132 a_{03} b_{02} b_{20} - 80 b_{02}^3 b_{20} -
     286 a_{02} b_{02} b_{11} b_{20} + 47 b_{02} b_{11}^2 b_{20} - 144 b_{02} b_{12} b_{20} -
     160 a_{02}^2 b_{20}^2 - 102 a_{03} b_{20}^2
     \\
     & + 12 b_{02}^2 b_{20}^2 - 306 a_{02} b_{11} b_{20}^2 +
     61 b_{11}^2 b_{20}^2 - 114 b_{12} b_{20}^2 + 260 b_{02} b_{20}^3 + 200 b_{20}^4 -
     30 a_{02} b_{02} b_{21} - 39 b_{02} b_{11} b_{21}
     \\
     & + 84 a_{02} b_{20} b_{21} - 96 b_{11} b_{20} b_{21} +
     27 b_{21}^2 + 132 a_{02}^2 b_{30} + 81 a_{03} b_{30} - 66 b_{02}^2 b_{30} + 207 a_{02} b_{11} b_{30} -
     6 b_{11}^2 b_{30}
     \\
     & - 6 a_{02}^2 b_{11}^2 + 81 b_{12} b_{30} - 498 b_{02} b_{20} b_{30} - 498 b_{20}^2 b_{30} + 135 b_{30}^2,
\\
j_2 = & 224 a_{02}^3 b_{02} + 240 a_{02} a_{03} b_{02} - 16 a_{02} b_{02}^3 - 184 a_{02}^2 b_{02} b_{11} +
     6 a_{03} b_{02} b_{11} + 40 b_{02}^3 b_{11} + 124 a_{02} b_{02} b_{11}^2 - 11 b_{02} b_{11}^3
     \\
     &-     156 a_{02} b_{02} b_{12} + 54 b_{02} b_{11} b_{12} - 120 a_{02}^3 b_{20} - 108 a_{02} a_{03} b_{20} +
     104 a_{02} b_{02}^2 b_{20} + 40 a_{02}^2 b_{11} b_{20} + 27 a_{03} b_{11} b_{20}
     \\
     &+     38 b_{02}^2 b_{11} b_{20} + 77 a_{02} b_{11}^2 b_{20} - 8 b_{11}^3 b_{20} + 24 a_{02} b_{12} b_{20} +
     39 b_{11} b_{12} b_{20} + 140 a_{02} b_{02} b_{20}^2 - 64 b_{02} b_{11} b_{20}^2
     \\
     & - 120 a_{02} b_{20}^3 - 50 b_{11} b_{20}^3 - 48 a_{02}^2 b_{21} - 45 a_{03} b_{21} - 42 b_{02}^2 b_{21} -
     87 a_{02} b_{11} b_{21} - 27 b_{12} b_{21} - 6 b_{02} b_{20} b_{21}
     \\
     & + 6 b_{11}^2 b_{21}+ 30 b_{20}^2 b_{21} - 270 a_{02} b_{02} b_{30} + 105 b_{02} b_{11} b_{30} + 108 a_{02} b_{20} b_{30} +
     84 b_{11} b_{20} b_{30} - 36 b_{21} b_{30}.
\end{aligned}
\end{equation*}
}

\end{document}